\documentclass[12pt]{amsart}
\usepackage{a4}
\usepackage{amssymb}

\makeatletter
\@addtoreset{equation}{section}
\makeatother

\usepackage{color}
\usepackage[mathscr]{eucal}
\usepackage{amsmath,amsthm,amssymb}
\usepackage{mathrsfs}
\usepackage{enumerate}
\usepackage{bm}
\usepackage{graphicx}
\usepackage{verbatim}
\usepackage{wrapfig}
\usepackage{ascmac}
\usepackage{multicol}

\usepackage[dvipdfmx]{hyperref}

\newtheorem{Prop}{Proposition}[section]
\newtheorem{Thm}[Prop]{Theorem}
\newtheorem{Lem}[Prop]{Lemma}

\newtheorem{Ex}[Prop]{Example}

\theoremstyle{definition}
\newtheorem{Def}[Prop]{Definition}

\newcommand{\R}{{\mathbb R}}

\newcommand{\Z}{{\mathbb Z}}
\newcommand{\GL}{\mathrm{GL}}
\newcommand{\SO}{\mathrm{SO}}
\newcommand{\OO}{\mathrm{O}}

\newcommand{\inner}[2]{\langle #1 , #2 \rangle} 
\newcommand{\Span}[1]{\mathrm{Span} \{ #1 \}}

\newcommand{\Map}{\mathrm{Map}}
\newcommand{\id}{\mathrm{id}}
\newcommand{\Aut}{\mathrm{Aut}}

\newcommand{\Inn}{\mathrm{Inn}}
\newcommand{\Dis}{\mathrm{Dis}}

\newcommand{\rnum}[1]{\expandafter{\romannumeral #1}}

\title[Homogeneous quandles with abelian inner automorphism groups]{Homogeneous quandles with abelian inner automorphism groups and vertex-transitive graphs}

\dedicatory{Dedicated to Professor Bang-Yen Chen's 80th birthday} 

\author{Konomi Furuki} 
\address[K.~Furuki]{Department of Mathematics, Hiroshima University, 
Higashi-Hiroshima, Japan 739-8526} 

\author{Hiroshi Tamaru}
\address[H.~Tamaru]{Department of Mathematics, Osaka Metropolitan University, Osaka City, Japan 558-8585} 
\email{tamaru@omu.ac.jp}

\thanks{
This work was supported by JSPS KAKENHI Grant Numbers JP19K21831, JP22H01124. 
This work was partly supported by MEXT Promotion of Distinctive Joint Research Center Program JPMXP0723833165.
} 

\date{}

\makeatletter
\@namedef{subjclassname@2020}{\textup{2020} Mathematics Subject Classification}
\makeatother


\begin{document} 

\maketitle 

\begin{abstract} 
A quandle is an algebraic system originated in knot theory, 
and can be regarded as a generalization of symmetric spaces. 
The inner automorphism group of a quandle is defined as the group generated by the point symmetries 
(right multiplications). 
In this paper, 
starting from any simple graphs, 
we construct quandles whose inner automorphism groups are abelian. 
We also prove that the constructed quandle is homogeneous if and only if the graph is vertex-transitive. 
This shows that there is a wide family of quandles with abelian inner automorphism groups, 
even if we impose the homogeneity. 
The key examples of such quandles are realized as subquandles of oriented real Grassmannian manifolds. 
\end{abstract}

\section{Introduction}
\label{sec:1} 

In the theory of symmetric spaces, 
the contributions by Chen and Nagano are immeasurable. 
They introduced several important notions, 
such as antipodal subsets and two-numbers (\cite{CN}), among many others. 
These notions are known to be related closely with several important areas in mathematics (see \cite{Chen}). 
We have studied quandles from the viewpoint of symmetric spaces. 
In \cite{KNOT}, 
the notion of antipodal subsets is transferred to quandles, 
and a generalization called $s$-commutative subsets is given. 
This paper is a natural continuation. 
In particular, 
we construct many examples of finite quandles with nice properties, 
owing to the theory of symmetric spaces developed by Chen and Nagano. 

The notion of quandles is originated in knot theory (\cite{Joyce, Matveev}). 
Thereafter quandles have been studied very actively, 
and recently the notion of quandles is recognized as an important concept in many branches of mathematics. 
Among others, quandles can be regarded as a generalization of symmetric spaces. 
In fact, the definition of quandles can be formulated analogously to symmetric spaces as follows. 
Let $X$ be a set, and denote by $\Map(X,X)$ the set of all maps from $X$ to itself. 
Then a pair $(X,s)$ with $s: X \rightarrow \Map(X,X)$ is called a \textit{quandle} if 
\begin{itemize}
\item[(Q1)]
for any $ x \in X $, $ s_{x}(x) = x $; 
\item[(Q2)]
for any $ x \in X $, $ s_{x} $ is bijective; 
\item[(Q3)]
for any $ x,y \in X $, $ s_{x} \circ s_{y} = s_{s_{x}(y)} \circ s_{x}$.
\end{itemize}
Each $s_x$ is called the \textit{point symmetry} at $x$. 
Note that these three conditions correspond to the Reidemeister moves in knot theory. 
For more details and related results, 
we refer to \cite{Carter, St} (see also \cite{Kamada}) and references therein. 

For symmetric spaces, 
flat ones are the most basic examples, 
and play fundamental roles in the structure theory (as is seen in the theory of maximal tori). 
Here, the flatness means that the curvature tensor vanishes identically. 
Recall that a Riemannian symmetric space is flat if and only if 
the group $G^0(X,s)$ generated by the compositions of two symmetries $s_x \circ s_y$ is abelian (\cite{Loos}). 
Therefore, 
quandles with abelian or ``small'' transformation groups would play similar important roles 
in a possible structure theory of quandles. 
For a quandle $(X,s)$, 
several transformation groups have been studied, 
such as the automorphism group $\Aut(X,s)$ and the following subgroups: 
\begin{align} 
\label{eq:0818} 
\Dis(X,s) \subset G^0(X,s) \subset \Inn(X,s) \subset \Aut(X,s) . 
\end{align} 
Recall that $\Inn(X,s)$ denotes the group generated by all point symmetries $s_x$, 
called the inner automorphism group. 
The group of displacements $\Dis(X,s)$ is generated by 
the elements of the form $s_x \circ s_y^{-1}$. 
A quandle $(X,s)$ is said to be \textit{flat} if $G^0(X,s)$ is abelian (\cite{IT, Singh}), 
and \textit{medial} if $\Dis(X,s)$ is abelian 
(\cite{Hulpke-Stanovsky-Vojtechovsky, Jedlicka-Pilitowska-Stanovsky-ZamojskaDzienio}). 

To the contrary, 
quandles with ``large'' transformation groups have also been studied. 
A quandle $(X,s)$ is said to be \textit{homogeneous} (resp.~\textit{connected}) 
if $\Aut(X,s)$ (resp.~$\Inn(X,s)$) acts transitively on $X$. 
A quandle $(X,s)$ is said to be \textit{two-point homogeneous} if 
$\Inn(X,s)$ acts doubly-transitively on $X$ (\cite{K.T.W., Tamaru, Vendramin, Wada}). 
A similar notion with respect to $\Aut(X,s)$ has also been studied (\cite{Bonatto}). 

In this paper, 
we study quandles $(X,s)$ whose inner automorphism groups $\Inn(X,s)$ are abelian. 
These quandles are flat and medial by (\ref{eq:0818}), 
but still form a large class. 
The results of this paper yield that there is a wide family of quandles with abelian inner automorphism groups, 
even if we impose the homogeneity. 
In other words, 
there are many quandles such that $\Inn(X,s)$ are small but $\Aut(X,s)$ are large. 
This would be in contrast to the following result. 

\begin{Prop}[\cite{IT}] 
\label{prop:IT}
Every finite flat connected quandle must be a discrete torus, 
that is a direct product of dihedral quandles, 
with odd cardinality. 
\end{Prop} 

The condition in this proposition means that $G^0(X,s)$ is small but $\Inn(X,s)$ is large. 
It seems to be natural that this kind of ``unbalanced'' condition is restrictive, 
but this is not true for our case. 

The first result of this paper constructs examples of homogeneous quandles with abelian inner automorphism groups. 
The key idea of this construction comes from symmetric spaces. 
We consider the oriented real Grassmannian manifold $G_k(\R^n)^\sim$, 
consisting of all oriented $k$-dimensional linear subspaces in $\R^n$. 
This is naturally a symmetric space, and hence a quandle (see Section~\ref{sec3} for details). 
As subquandles in $G_k(\R^n)^\sim$, 
we can construct interesting examples of quandles as follows. 

\begin{Thm} 
Let $\{ e_1 , \ldots , e_n \}$ be the standard basis of $\R^n$, 
and $A(k,n)$ be the set of all oriented $k$-dimensional linear subspaces spanned by $\{ e_{i_1} , \ldots , e_{i_k} \}$. 
Then $A(k,n)$ is a subquandle of $G_k(\R^n)^\sim$. 
Furthermore, 
$A(k,n)$ is a homogeneous disconnected quandle, 
and its inner automorphism group is abelian. 
\end{Thm} 

This example would be interesting also from the viewpoint of symmetric spaces 
(for recent studies on particular subsets in symmetric spaces, see \cite{KNOT} and references therein). 
On the other hand, 
one can observe that the structures of the above quandles $A(k,n)$ can be described in terms of graphs. 
The second result of this paper generalizes this construction, 
namely, we construct quandles from graphs. 
Let $G = (V(G), E(G))$ be a simple graph, 
where $V(G)$ denotes the set of vertices and $E(G)$ the set of edges. 
Let $\Z_2 := \Z / 2 \Z$, 
and we use the adjacent function $e : V(G) \times V(G) \to \Z_2$ defined by 
\begin{align} 
\label{eq:adjacent} 
e(v,w) := \left\{ 
\begin{array}{ll}
1 & (\mbox{if $v$ and $w$ are joined by an edge}) , \\ 
0 & (\mbox{otherwise}) . 
\end{array} 
\right. 
\end{align} 

\begin{Thm} 
\label{thm:1-3} 
Let $G = (V(G), E(G))$ be a simple graph, and put $X := V(G) \times \Z_2$. 
We define a map $s : X \to \Map(X,X)$ by 
\begin{align*} 
s_{(v , a)} (w , b) := (w , b + e(v,w) ) . 
\end{align*} 
Then $Q_G := (X,s)$ is a disconnected quandle, 
and its inner automorphism group is abelian. 
Furthermore, 
$Q_G$ is homogeneous if and only if the graph $G$ is vertex-transitive. 
\end{Thm} 

This shows that there is a wide family of homogeneous quandles with abelian inner automorphism groups. 
Note that the above construction can be regarded as a special case of abelian extensions using quandle $2$-cocycles 
(see Section~\ref{sec4}). 
Finally, we note that the quandles $Q_G$ constructed from graphs $G$ are always involutive, that is, 
all point symmetries $s_x$ are involutive. 
There must exist non-involutive homogeneous quandles with abelian inner automorphism groups, 
which will be studied in the forthcoming paper. 

\section{Preliminaries}
\label{sec2}

In this section, 
we recall some necessary notions of quandles, 
and also describe fundamental examples. 
Two particular finite quandles, given as subsets in the spheres, are studied in detail. 

\subsection{Examples of quandles} 

As mentioned in Section~\ref{sec:1}, we denote a quandle by $(X,s)$, 
where the map $s : X \to \Map(X,X)$ is called a \textit{quandle structure}. 
In this subsection, we give some examples of quandles. 

\begin{Ex} 
Let $X$ be any nonempty set, and define $s$ by $s_{x} = \id_{X}$ for any $x \in X$. 
Then $(X,s)$ is a quandle, which is called a \textit{trivial quandle}. 
\end{Ex} 

Note that $\id_{X}$ denotes the identity map. 
We need examples of nontrivial quandles. 
First of all, we recall properties of reflections with respect to linear subspaces, 
which will be used to define some quandle structures. 
Denote by $\OO(n)$ the orthogonal group of $\R^n$. 

\begin{Lem} 
\label{lem:Grassmann} 
Let $V$ be a linear subspace in $\R^n$, 
and denote by $r_V$ the reflection on $\R^n$ with respect to $V$. 
Then one has 
\begin{enumerate} 
\item[$(1)$] 
$r_V \in \OO(n)$$;$ 
\item[$(2)$] 
for any $g \in \OO(n)$, it satisfies $g \circ r_V = r_{g(V)} \circ g$. 
\end{enumerate} 
\end{Lem} 

\begin{proof} 
Let $\inner{}{}$ be the canonical inner product on $\R^n$, 
and denote by $V^\perp$ the orthogonal complement of $V$ in $\R^n$ with respect to $\inner{}{}$. 
Then, by definition, the reflection $r_V$ satisfies 
\begin{align} 
r_V (x + y) = x - y \quad ( x \in V , \ y \in V^\perp ) . 
\end{align} 
Therefore it is easy to show that $r_V$ preserves $\inner{}{}$, 
that is, $r_V \in \OO(n)$. 
This shows the first assertion. 
In order to prove the second assertion, take any $g \in \OO(n)$. 
Then one can see that 
$g \circ r_V$ and $r_{g(V)} \circ g$ coincide with each other on $V$, and also on $V^\perp$. 
Since both maps are linear, this completes the proof. 
\end{proof} 

By using the reflections, 
we can define a quandle structure on the sphere $S^{n-1}$, 
to be exact, the unit sphere in $\R^n$ centered at the origin. 
The following quandle $(S^{n-1},s)$ is called the \textit{sphere quandle} in this paper. 

\begin{Ex} 
Let $S^{n-1}$ be the unit sphere in $\R^n$. 
For each $x \in S^{n-1}$, 
let us define $s_x$ by the reflection with respect to $\mathrm{Span} \{ x \}$. 
Then this $s$ gives a quandle structure on $S^{n-1}$. 
\end{Ex} 

\begin{proof} 
For each $x \in S^{n-1}$, the quandle structure is defined as 
\begin{align} 
s_x : S^{n-1} \to S^{n-1} : y \mapsto r_{ \mathrm{Span} \{ x \} } (y) . 
\end{align} 
By definition, one can directly see that $s$ is well-defined, 
that is, $s_x$ preserves $S^{n-1}$. 
Conditions~(Q1) and (Q2) are easy to check. 
Condition~(Q3) follows from Lemma~\ref{lem:Grassmann}. 
\end{proof} 

In order to describe more examples of quandles, 
we use the following notion of subquandles. 

\begin{Def} 
Let $(X,s)$ be a quandle. 
A nonempty subset $A \subset X$ is called a \textit{subquandle} if 
$s_a^{\pm 1}(A) \subset A$ holds for every $a \in A$. 
\end{Def} 

It is easy to see that a subquandle is a quandle. 
Subquandles provide several interesting examples of quandles. 

\begin{Ex} 
Let $r \in \Z_{>0}$, and $(S^1, s)$ be the sphere quandle with $n=1$. 
Then the following $D_r$ is a subquandle$:$ 
\begin{align} 
D_r := \{ (\cos (2 k \pi/r) , \sin (2 k \pi/r)) \in S^1 \mid k \in \Z \} . 
\end{align} 
\end{Ex} 

This quandle $D_r$ is called the \textit{dihedral quandle} of order $r$, which is well-known. 
The following is also easy, but provides one of key examples. 

\begin{Ex} 
\label{ex:our_quandle}
Let $(S^{n-1}, s)$ be the sphere quandle, 
and $\{ e_1, \ldots, e_n \}$ be the standard basis of $\R^n$. 
Then the following $A^{n-1}$ is a subquandle$:$ 
\begin{align} 
A^{n-1} := \{ \pm e_1 , \ldots , \pm e_n \} . 
\end{align} 
\end{Ex} 

\begin{proof} 
Since $s_x$ is involutive, 
we have only to show that $s_a(A^{n-1}) \subset A^{n-1}$ for all $a \in A^{n-1}$. 
By the definition of $s$, one can easily see that 
\begin{align} 
\label{eq:A_relation} 
s_{e_i} = s_{- e_i} , \quad 
s_{e_i} (\pm e_i) = \pm e_i , \quad 
s_{e_i} (\pm e_j) = \mp e_j \ (\mbox{for $i \neq j$}) . 
\end{align} 
This completes the proof of this example. 
\end{proof} 

Note that $A^0 = \{ \pm e_1 \}$ is a trivial quandle, 
and $A^1 = \{ \pm e_1 , \pm e_2 \}$ coincides with the dihedral quandle $D_4$ of order $4$. 
Properties of the quandle $A^{n-1}$ will be mentioned in the following subsections. 

\subsection{Homogeneous quandles} 
\label{subsec:homogeneous-quandles}

In this subsection, we recall the definition of homogeneous quandles, 
and mention some examples. 

\begin{Def}\label{Def2.3}
Let $ (X,s^{X}) $ and $ (Y,s^{Y}) $ be quandles. 
Then $ f : X \rightarrow Y $ is called a \textit{homomorphism} if for any $ x \in X $, it satisfies 
\begin{align} 
f \circ s^{X}_{x} = s^{Y}_{f(x)} \circ f . 
\end{align} 
\end{Def} 

A bijective homomorphism is called an \textit{isomorphism}. 
An isomorphism from a quandle $ (X,s) $ onto $ (X,s) $ itself is called an \textit{automorphism} of $ (X,s) $. 

\begin{Def} 
Let $(X,s)$ be a quandle. 
The group of all automorphisms of $(X,s)$ is called the 
\textit{automorphism group} of $(X,s)$, and denoted by $\Aut(X,s)$. 
A quandle $(X,s)$ is said to be \textit{homogeneous} if $\Aut(X,s)$ acts transitively on $X$. 
\end{Def} 

For the automorphism group of a subquandle, one can easily see the following. 
The proof is an easy exercise. 

\begin{Lem} 
\label{lem:subquandle} 
Let $(X,s)$ be a quandle, and $A$ be a subquandle of $(X,s)$. 
Then the following $N(A)$ acts on $A$ as automorphisms$:$ 
\begin{align} 
N(A) := \{ f \in \Aut(X,s) \mid f(A) = A \} . 
\end{align} 
\end{Lem} 

We here describe examples of homogeneous quandles. 
In fact, the examples mentioned in the previous subsection are all homogeneous. 

\begin{Ex} 
\label{ex:homogeneous} 
The dihedral quandle $D_r$ and 
the quandle $A^{n-1} = \{ \pm e_1 , \ldots , \pm e_n \}$ are both homogeneous. 
\end{Ex} 

\begin{proof} 
One knows $\OO(n) \subset \Aut(S^{n-1},s)$ by Lemma~\ref{lem:Grassmann}, 
and hence we can use Lemma~\ref{lem:subquandle} for both cases. 
For the dihedral quandle $D_r$, 
one can see that $N(D_r)$ acts transitively on $D_r$, 
since it contains a rotation of angle $2 \pi / r$. 
Similarly one can also show that $N(A^{n-1})$ acts transitively on $A^{n-1}$, 
since it contains a rotation of angle $\pi / 2$ in the $e_i e_j$-plane, for any $i,j$. 
\end{proof} 

\subsection{Connected quandles}

In this subsection, 
we recall the notion of connected quandles and connected components. 
We also describe examples of connected and disconnected quandles. 

\begin{Def} 
Let $(X,s)$ be a quandle. 
The group generated by $\{ s_{x} \mid x \in X \}$ is called the 
\textit{inner automorphism group} of $(X,s)$, and denoted by $\Inn(X,s)$. 
A quandle 
$ (X,s) $ is said to be \textit{connected} if $ \Inn(X,s) $ acts transitively on $ X $. 
\end{Def}

It follows from (Q2) and (Q3) that $s_{x} \in \Aut(X,s)$ for all $x \in X$. 
Then one has $\Inn(X,s) \subset \Aut(X,s)$, which yields the following. 

\begin{Prop} 
If a quandle $(X,s)$ is connected, then it is homogeneous. 
\end{Prop} 

For a quandle $(X,s)$ and $x \in X$, the orbit $\Inn(X,s).x$ is called the 
\textit{connected component} containing $x$. 
It is easy to see that every connected component is a subquandle. 
Note that a quandle $(X,s)$ is connected if and only if any connected component coincides with $X$. 

\begin{Ex} 
\noindent 
\begin{enumerate} 
\item[$(1)$] 
The dihedral quandle $D_r$ is connected if and only if $r$ is odd$;$ 
\item[$(2)$] 
The quandle $A^{n-1} = \{ \pm e_1 , \ldots , \pm e_n \}$ is always disconnected. 
\end{enumerate} 
\end{Ex} 

\begin{proof} 
The assertion for $D_r$ is well-known. 
We have only to show the second assertion. 
If $n=1$, then $A^0 = \{ \pm e_1 \}$ is trivial and hence disconnected. 
If $n \geq 2$, 
then it is easy to see from (\ref{eq:A_relation}) that $\{ \pm e_i \}$ are connected components, 
and hence $A^{n-1}$ is disconnected. 
\end{proof} 

Finally in this subsection, 
we show that an automorphism maps a connected component onto a connected component. 

\begin{Prop} 
\label{prop:2-14} 
Let $\varphi \in \Aut(X,s)$, and take $x \in X$. 
Then $\varphi$ maps the connected component containing $x$ onto the connected component containing $\varphi(x)$. 
\end{Prop} 

\begin{proof} 
First of all, 
we recall that the inner automorphism group $\Inn(X,s)$ is a normal subgroup of $\Aut(X,s)$. 
This follows from 
\begin{align} 
f \circ s_y \circ f^{-1} = s_{f(y)} 
\end{align} 
for any $f \in \Aut(X,s)$ and $y \in X$. 
Note that the connected component containing $x$ is $\Inn(X,s).x$. 
We thus have 
\begin{align} 
\varphi (\Inn(X,s).x) 
= \varphi \Inn(X,s) \varphi^{-1}.(\varphi(x)) 
= \Inn(X,s).(\varphi(x)) , 
\end{align} 
which is the connected component containing $\varphi(x)$. 
\end{proof} 

\subsection{Commutativity of transformation groups} 

In this subsection, 
We study the commutativity of $G^0(X,s)$ and $\Inn(X,s)$ for the quandles described above. 
Recall that the definition of flat quandles is given in \cite{IT}. 

\begin{Def} 
A quandle $(X,s)$ is said to be \textit{flat} if 
the group $G^0(X,s)$ generated by $\{ s_{x} \circ s_y \mid x, y \in X \}$ is abelian. 
\end{Def}

We are interested in flat quandles, 
and also in quandles with $\Inn(X,s)$ being abelian. 
Recall that flat connected quandles have been classified in \cite{IT, Singh}. 

\begin{Ex} 
\label{ex:flat}
\noindent 
\begin{enumerate} 
\item[$(1)$] 
The dihedral quandle $D_r$ is always flat, 
and it has an abelian inner automorphism group if and only if $r \in \{ 1,2,4 \}$$;$ 
\item[$(2)$] 
The quandle $A^{n-1} = \{ \pm e_1 , \ldots , \pm e_n \}$ has an abelian inner automorphism group, 
for every $n \in \Z_{\geq 1}$. 
\end{enumerate} 
\end{Ex} 

\begin{proof} 
We show (1). 
For the dihedral quandle $D_r$, 
recall that $s_x$ is defined by the reflection $r_{ \Span{x} }$. 
Note that every reflection is an element in $\OO(2)$ with determinant $-1$. 
This yields that $G^0(D_r,s)$ is a subgroup of $\SO(2)$, 
and hence it is abelian, 
which shows that $D_r$ is flat. 
The assertion on $\Inn(D_r, s)$ follows from the following fact: 
for $x,y \in S^1$, 
two reflections $r_{\mathrm{Span} \{ x \}}$ and $r_{\mathrm{Span} \{ y \}}$ commute each other 
if and only if $x$ and $y$ are orthogonal or parallel. 
In particular, 
the inner automorphism group is isomorphic to $\Z_2 \times \Z_2$ if $n=4$. 

We show (2). 
For the quandle $A^{n-1}$, 
recall that $s_{\pm e_i}$ is defined by the reflection $r_{ \Span{e_i} }$, 
and this reflection can be represented as a diagonal matrices with respect to the basis 
$\{ e_1 , \ldots , e_n \}$. 
This yields that $r_{ \Span{e_i} }$ and $r_{ \Span{e_j} }$ commute, 
and hence $s_{e_i}$ and $s_{e_j}$ commute. 
\end{proof} 

We conclude that the quandle 
$A^{n-1} = \{ \pm e_1 , \ldots , \pm e_n \}$ 
is homogeneous, disconnected, and having abelian inner automorphism group. 
In fact, $A^{n-1}$ are simplest examples of the quandles constructed in this paper. 

\section{Motivating examples} 
\label{sec3} 

The quandle 
$A^{n} := \{ \pm e_1 , \ldots , \pm e_{n+1} \}$ 
defined in the previous section is homogeneous, disconnected, 
and having an abelian inner automorphism group. 
Recall that $A^n$ is defined as a subquandle of the sphere quandle $S^n$. 
In this section, we introduce further examples of quandles with the same properties. 
They are given by certain subquandles of the oriented real Grassmannians $G_k(\R^n)^\sim$ 
(also called the oriented real Grassmannian manifolds). 

\subsection{Real Grassmannians} 

In this subsection, 
we briefly recall (nonoriented) real Grassmannians and their quandle structures. 
As a set, real Grassmannians are defined by 
\begin{align} 
G_k(\R^n) := \{ V \mid \mbox{a linear subspace in $\R^n$ with $\dim V = k$} \} . 
\end{align} 
Recall that $r_V$ denotes the reflection with respect to a linear subspace $V$ in $\R^n$ 
(see Lemma~\ref{lem:Grassmann}). 
The quandle structure on $G_k(\R^n)$ is induced from the reflections. 

\begin{Prop} 
For each $V \in G_k(\R^n)$, 
denote by $s_V : G_k(\R^n) \to G_k(\R^n)$ the map induced from $r_V$. 
Then this $s$ gives a quandle structure on $G_k(\R^n)$. 
\end{Prop} 

\begin{proof} 
Note that the induced map $s_V$ is given as 
\begin{align} 
s_V (W) := \{ r_V (w) \in \R^n \mid w \in W \} . 
\end{align} 
It is easy to see $s_V (W) \in G_k(\R^n)$, that is, $s$ is well-defined. 
By the definition of $s$, Conditions~(Q1) and (Q2) are easy to check. 
Condition~(Q3) follows from Lemma~\ref{lem:Grassmann}. 
\end{proof} 

It also follows directly from Lemma~\ref{lem:Grassmann}~(2) that 
$\OO(n)$ acts on $G_k(\R^n)$ as automorphisms. 

\subsection{Oriented real Grassmannians} 

In this subsection, 
we recall the definition and some basic facts on the oriented real Grassmannians $G_k(\R^n)^\sim$. 
First of all, 
let us recall an orientation of a vector space. 

\begin{Def} 
Let $V$ be a $k$-dimensional real vector space, 
and let $( v_1, \ldots, v_k )$ and $( w_1, \ldots, w_k )$ be ordered bases of $V$. 
Then they are said to have the 
\textit{same orientation} if there exists $g \in \GL(k,\R)$ with $\det(g) > 0$ such that 
\begin{align} 
( v_1, \ldots, v_k ) = ( w_1, \ldots, w_k ) g . 
\end{align} 
\end{Def} 

Note that the matrix $g$ is so-called the change-of-basis matrix. 
It is clear that the above ``same orientation'' gives an equivalence relation. 

\begin{Def} 
Each element of the following quotient space $\mathrm{OR}(V)$ is called an \textit{orientation} of $V$$:$ 
\begin{align} 
\mathrm{OR}(V) := \{ \mbox{ordered bases of $V$} \} / \mbox{``same orientation''} . 
\end{align} 
\end{Def} 

One can easily see that $\mathrm{OR}(V)$ consists of two elements, 
that is, there exist exactly two orientations of $V$. 
An orientation is denoted as $\sigma = [ (v_1 , \ldots , v_k) ]$, 
where $(v_1 , \ldots , v_k)$ is an ordered basis. 

\begin{Def} 
The set of all oriented $k$-dimensional linear subspaces in $\R^n$ 
is called the \textit{oriented real Grassmannian}, and denoted by 
\begin{align} 
G_k(\R^n)^\sim := \{ (V, \sigma) \mid V \in G_k(\R^n) , \ \sigma \in \mathrm{OR}(V) \} . 
\end{align} 
\end{Def} 

As in the nonoriented case, 
the quandle structure is induced from the reflections $r_V$. 
First of all, we show that $\OO(n)$ acts on $G_k(\R^n)^\sim$. 

\begin{Lem} 
\label{lem:action_oriented} 
The orthogonal group $\OO(n)$ acts on $G_k(\R^n)^\sim$, 
where the action of $g \in \OO(n)$ on $(W , \tau) \in G_k(\R^n)^\sim$ with $\tau = [(w_1, \ldots, w_k)]$ 
is defined by 
\begin{align} 
g . (W, \tau) := ( g(W) , [ (g(w_1) , \ldots , g(w_k)) ]) . 
\end{align} 
\end{Lem} 

\begin{proof} 
One knows $g(W) \in G_k(\R^n)$. 
Furthermore, if 
$(w_1 , \ldots , w_k)$ and $(w'_1 , \ldots , w'_k)$ have the same orientation, 
then so do $(g(w_1) , \ldots , g(w_k))$ and $(g(w'_1) , \ldots , g(w'_k))$. 
This shows that the action is well-defined. 
One also needs to show the conditions of the group actions, 
that is, 
\begin{align} 
e.(W, \tau) = (W, \tau) , \quad h.(g.(W, \tau)) = (hg).(W, \tau) , 
\end{align} 
where $e$ is the identify, and for any $g,h \in \OO(n)$. 
Both of them follow directly from the definition. 
\end{proof} 

Recall that the reflection satisfies $r_V \in \OO(n)$. 
Therefore, by this lemma, $r_V$ induces a map from $G_k(\R^n)^\sim$ onto itself. 
This gives a quandle structure. 

\begin{Prop} 
\label{prop:oriented} 
For $(V , \sigma) \in G_k(\R^n)^\sim$, 
let us denote by $s_{(V, \sigma)}$ the action on $G_k(\R^n)^\sim$ induced from the reflection $r_V$. 
Then this $s$ gives a quandle structure on $G_k(\R^n)^\sim$. 
Furthermore, $\OO(n)$ acts on $G_k(\R^n)^\sim$ as automorphisms. 
\end{Prop} 

\begin{proof} 
Conditions~(Q1) and (Q2) are obvious. 
It follows from Lemma~\ref{lem:Grassmann} that $s$ satisfies (Q3), 
and also $\OO(n)$ acts on $G_k(\R^n)^\sim$ as automorphisms. 
\end{proof} 

It is well-known that $G_k(\R^n)^\sim$ is a Riemannian symmetric space, 
whose geodesic symmetry at $(V, \sigma)$ is nothing but the above $s_{(V, \sigma)}$. 

\subsection{Some subquandles in the oriented real Grassmannians} 
\label{subsection:some_subquandles} 

In this subsection, 
we introduce some subquandles in $G_k(\R^n)^\sim$, 
and prove that they are homogeneous, disconnected, and having abelian inner automorphism groups. 
For simplicity of the notation, let us put 
\begin{align} 
(i_1 , \ldots , i_k) := \left( \mathrm{Span} \{ e_{i_1} , \ldots e_{i_k} \} , 
[ (e_{i_1} , \ldots e_{i_k}) ] \right) \in G_k(\R^n)^\sim . 
\end{align} 
We also denote by $- (i_1 , \ldots , i_k)$ the same linear subspace, 
but equipped with the opposite orientation. 

\begin{Prop} 
The following $A(k,n)$ is a subquandle of $G_k(\R^n)^\sim$$:$ 
\begin{align} 
A(k,n) := \{ \pm ( {i_1} , \ldots , {i_k} ) \in G_k(\R^n)^\sim \mid 1 \leq i_1 < \cdots < i_k \leq n \} . 
\end{align} 
\end{Prop} 

\begin{proof} 
Take any $I := \pm ( {i_1} , \ldots , {i_k} ) \in A(k,n)$. 
We denote by $r_I$ the reflection with respect to $\mathrm{Span} \{ e_{i_1} , \ldots e_{i_k} \}$. 
Then it satisfies 
\begin{align} 
\label{eq:reflection} 
r_{I} (e_j) = \left\{ 
\begin{array}{rl} 
e_j & (\mbox{if $j \in \{ i_1 , \ldots , i_k \}$}) , \\ 
- e_j & (\mbox{otherwise}) . 
\end{array} 
\right. 
\end{align} 
Recall that $s_I$ is induced from $r_I$. 
Then $s_{I}$ is involutive, since so is $r_{I}$. 
We thus have only to show that 
\begin{align} 
s_I (A(k,n)) \subset A(k,n) . 
\end{align} 
Take any $J := \pm ( {j_1} , \ldots , {j_k} ) \in A(k,n)$. 
Then we have 
\begin{align} 
\label{eq:IJ} 
s_{I} (J) = \left\{ 
\begin{array}{rl} 
J & (\mbox{if $\# (\{ j_1 , \ldots , j_k \} \setminus \{ i_1 , \ldots , i_k \})$ is even}) , \\ 
- J & (\mbox{if $\# (\{ j_1 , \ldots , j_k \} \setminus \{ i_1 , \ldots , i_k \})$ is odd}) . 
\end{array} 
\right. 
\end{align} 
This yields that $s_{I} (J) \in A(k,n)$, which completes the proof. 
\end{proof} 

One can see from (\ref{eq:IJ}) that $A(1,n) \cong A^{n-1}$. 
As for the case of $A^n$, 
the quandles $A(k,n)$ have the following nice properties. 

\begin{Prop} 
The quandle $A(k,n)$ is homogeneous, disconnected, and having an abelian inner automorphism group. 
\end{Prop} 

\begin{proof} 
First of all, we show the homogeneity. 
Note that $\OO(n)$ acts on $G_k(\R^n)^\sim$ as automorphisms by Proposition~\ref{prop:oriented}. 
Hence, by Lemma~\ref{lem:subquandle}, it is enough to prove that 
the normalizer $N(A(k,n))$ in $\OO(n)$ acts transitively on $A(k,n)$. 
Take any $I = \pm (i_1 , \ldots , i_k) \in A(k,n)$, 
and show that there exists $g \in N(A(k,n))$ such that $g(I) = (1, \ldots , k)$. 
First of all, 
similarly to the proof of Example~\ref{ex:homogeneous}, 
the normalizer $N(A(k,n))$ contains a rotation of angle $\pi/2$ in the $e_i e_j$-plane, 
for any $i,j$. 
Then there exists $h \in N(A(k,n))$ such that 
\begin{align} 
h (\Span{ e_{i_1} , \ldots , e_{i_k}} ) = \Span{ e_1 , \ldots , e_k } . 
\end{align} 
This yields that $h(I) = \pm (1 , \ldots , k)$. 
Here, let us put 
\begin{align} 
g := ( \pm 1 , 1 , \ldots , 1) h \in N(A(k,n)) , 
\end{align} 
where the double-signs correspond. 
Then we have $g(I) = (1, \ldots , k)$, 
which completes the proof of the homogeneity. 

We next show that it is disconnected. 
Assume that $A(k,n)$ is connected. 
Then, for every $J \in A(k,n)$, it follows from (\ref{eq:IJ}) that 
\begin{align} 
\{ \pm J \} \supset \mathrm{Inn}(A(k,n)) . J = A(k,n) \supset \{ \pm J \} . 
\end{align} 
We thus have $\# A(k,n) = 2$. 
However, such quandle must be trivial, which is a contradiction. 
This proves that $A(k,n)$ is disconnected. 

We finally study the inner automorphism group. 
It follows from (\ref{eq:reflection}) that, 
for every $I \in A(k,n)$, 
the reflection $r_{I}$ can be realized as a diagonal matrices. 
Therefore, 
by the definition of $s_I$, 
the inner automorphism group $\mathrm{Inn}(A(k,n))$ is abelian. 
This completes the proof. 
\end{proof} 

\subsection{A key observation} 
\label{subsection:observation} 

In this subsection, 
we observe that one can construct a graph from the quandle $A(k,n)$. 
Recall that a graph $G = (V(G), E(G))$ 
consists of the set of vertices $V(G)$ and the set of edges $E(G)$. 

We here describe an idea of our construction by using 
\begin{align} 
A(2,4) = \{ \pm (1,2) , \pm (1,3) , \pm (1,4) , \pm (2,3) , \pm (2,4) , \pm (3,4) \} . 
\end{align} 
One knows that $A(2,4)$ is decomposed into $6$ connected components, 
which are of the form $\{ \pm (i,j) \}$. 
Recall that a connected component is an orbit of the inner automorphism group. 
We denote by $[i,j] := \{ \pm (i,j) \}$. 

\begin{Ex} 
For $A(2,4)$, let us construct a graph as follows$:$ 
\begin{itemize} 
\item 
put $V(G) := \{ [i,j] \mid 1 \leq i < j \leq 4 \}$, 
that is, the connected components correspond to vertices$;$ 
\item 
join $[i , j]$ and $[k , \ell]$ by an edge if $s_{\pm (i,j)}$ acts nontrivially on $\{ \pm (k,\ell) \}$. 
\end{itemize} 
Then we obtain the graph in Fig.~\ref{fig:1}. 
\end{Ex} 

\begin{figure}[h] 
\centering 
\includegraphics[width=5cm]{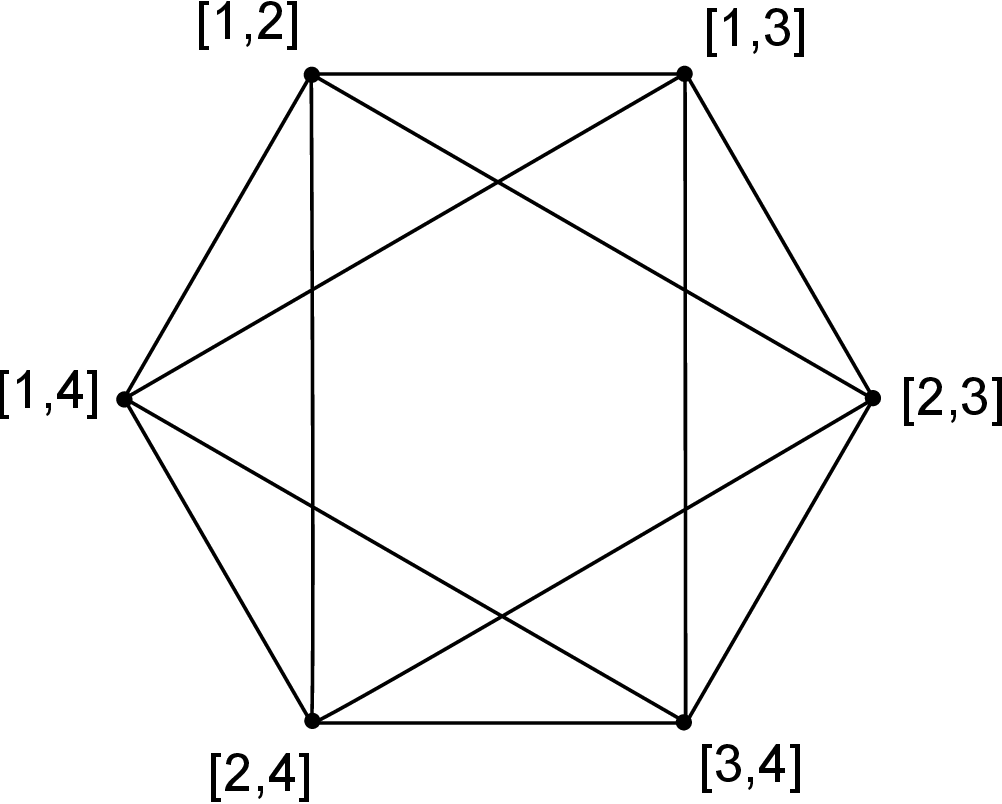} 
\caption{The graph associated from $A(2,4)$} 
\label{fig:1} 
\end{figure} 

\begin{proof} 
One can directly show the assertion by (\ref{eq:IJ}). 
Note that $s_{\pm (i,j)}$ acts nontrivially on $\{ \pm (k,\ell) \}$ 
if and only if 
$s_{\pm (k,\ell)}$ acts nontrivially on $\{ \pm (i,j) \}$. 
\end{proof} 

Conversely, 
one can expect that the quandle structure of $A(2,4)$ is reconstructed from the graph in Fig.~\ref{fig:1}. 
In fact, this is true. 
More generally, starting from any graphs, one can construct quandles. 
We will formulate this construction in the next sections. 

\section{The quandles associated to graphs} 
\label{sec4} 

By a simple graph, 
we mean an undirected graph containing no loops or multiple edges. 
In this section, 
we show the first part of Theorem~\ref{thm:1-3}. 
Namely, we construct quandles $Q_G$ associated to simple graphs $G$, 
and prove that they are always disconnected and having abelian inner automorphism groups. 

\subsection{Definition of $Q_G$} 

In this subsection, 
we define the quandle $Q_G$ for a simple graph $G$, 
by referring to the examples given in Subsection~\ref{subsection:observation}. 
We denote a graph by $G = (V(G), E(G))$, 
where $V(G)$ is the set of vertices and $E(G)$ is the set of edges. 
For each $v,w \in V(G)$, 
denote by $v \sim w$ if they are joined by an edge. 
We use the adjacent function $e : V(G) \times V(G) \to \Z_2$ defined in (\ref{eq:adjacent}). 

\begin{Def} 
Let $G = (V(G), E(G))$ be a simple graph, and put $X := V(G) \times \Z_2$. 
Then we define the \textit{associated map} $s : X \to \Map(X,X)$ by 
\begin{align} 
s_{(v , a)} (w , b) := (w , b + e(v,w) ) . 
\end{align} 
\end{Def} 

Intuitively, $Q_G$ is constructed by attaching $\Z_2$ on each vertex of a graph $G$. 
The map $s$ is defined by using information of the edges. 

\begin{Prop} 
Let $G = (V(G), E(G))$ be a simple graph. 
Then the associated map $s$ is a quandle structure on $X := V(G) \times \Z_2$. 
\end{Prop} 

\begin{proof} 
Since the graph $G$ is simple, 
one has $v \not \sim v$ and hence $e(v,v) = 0$. 
This proves (Q1). 
Also, by the definition of $s$, we have 
\begin{align} 
s_{(v,a)}^2 (w,b) = s_{(v,a)} (w , b + e(v,w)) = (w , b + 2 e(v,w)) = (w , b) . 
\end{align} 
This means $s_{(v , a)}^2 = \id$, and hence proves (Q2). 
It remains to show (Q3). 
By the definition of $s$, we have 
\begin{align} 
\label{eq:0822} 
s_{(v,a)} \circ s_{(w, b)} (u, c) = s_{(v,a)} (u, c + e(w,u)) = (u, c + e(w,u) + e(v,u)) . 
\end{align} 
On the other hand, since $s_{(v,0)} = s_{(v,1)}$ by definition, it satisfies 
\begin{align} 
\label{eq:3-3} 
s_{(s_{(v,a)} (w,b))} = s_{(w, b + e(v,w) )} = s_{(w, b)} . 
\end{align} 
This yields that 
\begin{align} 
s_{(s_{(v,a)} (w,b))} \circ s_{(v,a)} (u,c) = s_{(w,b)} (u, c + e(v,u)) = (u, c + e(v,u) + e(w,u)) , 
\end{align} 
which completes the proof. 
\end{proof} 

Therefore, the pair $(Q_G, s)$ is indeed a quandle for every graph $G$. 
This $(Q_G, s)$ is called the \textit{quandle associated to $G$} in this paper. 

\subsection{Examples of $Q_G$} 

As examples, we here describe $Q_G$ when $G$ are easy graphs. 
The first case is an empty graph. 
Recall that a graph is said to be \textit{empty} or \textit{edgeless} 
if there are no edges, that is, $E(G) = \emptyset$. 

\begin{Ex} 
\label{ex:trivial} 
A graph $G$ is empty if and only if $Q_G$ is a trivial quandle. 
\end{Ex} 

\begin{proof} 
It follows easily from the definition of the associated map $s$. 
\end{proof} 

The next easy example is given by a complete graph. 
Recall that a graph is said to be \textit{complete} if any two vertices are joined by an edge. 

\begin{Ex} 
If $G$ is a complete graph with $n$ vertices, 
then $Q_G$ is isomorphic to the quandle $A^{n-1} = \{ \pm e_1 , \ldots , \pm e_n \}$ described in 
Example~\ref{ex:our_quandle}. 
\end{Ex} 

\begin{proof} 
We denote by $V(G) = \{ v_1 , \ldots , v_n \}$. 
Then one can directly see that 
\begin{align} 
f : V(G) \times \Z_2 \to A^{n-1} : (v_i , a) \mapsto (-1)^a e_i 
\end{align} 
gives a quandle isomorphism between $Q_G$ and $A^{n-1}$. 
\end{proof} 

In Subsection~\ref{subsection:some_subquandles}, 
the subquandle $A(k,n)$ of the oriented real Grassmannian $G_k(\R^n)^\sim$ is defined. 
We here show that $A(k,n)$ is constructed from a graph. 

\begin{Ex} 
Let $N := \{ 1 , \ldots , n \}$ and fix $k \in N$. 
Let us define a graph $G$ by 
\begin{align*} 
\begin{split} 
V(G) & := \{ v \subset N \mid \# v = k \} , \\ 
v \sim w & : \Leftrightarrow \mbox{$\# (v \setminus w)$ is odd} . 
\end{split} 
\end{align*} 
Then the associated quandle $Q_G$ is isomorphic to $A(k,n)$. 
\end{Ex} 

\begin{proof} 
We construct a map $f : V(G) \times \Z_2 \to A(k,n)$. 
Let $v \in V(G)$, 
and we use the notation defined in Subsection~\ref{subsection:some_subquandles}. 
Then one has $v = \{ i_1 , \ldots , i_k \}$ with $i_1 < \cdots < i_k$. 
In terms of this expression, we define 
\begin{align} 
f ( \{ i_1 , \ldots , i_k \} , a ) := 
\left\{ 
\begin{array}{rl} 
(i_1 , \ldots , i_k) & (\mbox{if $a = 0$}) , \\ 
- (i_1 , \ldots , i_k) & (\mbox{if $a = 1$}) . 
\end{array} 
\right. 
\end{align} 
It is obvious that $f$ is bijective, 
and it follows from (\ref{eq:IJ}) that $f$ is a quandle homomorphism. 
\end{proof} 

The above graphs with $k=1$ are complete graphs, 
and those with $k=2$ are known as the \textit{Johnson graphs} $J(n,2)$. 
Recall that the Johnson graph $J(n,k)$ is defined by the same set of vertices $V(G)$, 
and $v, w \in V(G)$ are joined by an edge if $\# (v \cap w) = k-1$. 

\subsection{Properties of $Q_G$} 

In this subsection, 
we study some basic properties of the quandles $Q_G$ associated to graphs $G$. 
In particular, they are all disconnected and having abelian inner automorphism groups. 
We begin with fundamental properties. 

\begin{Prop} 
\label{prop:QG1} 
For every graph $G$, the quandle $Q_G$ satisfies the following$:$
\begin{enumerate} 
\item[$(1)$] 
all connected components consist of at most two points$;$ 
\item[$(2)$] 
it is crossed, that is, for all $x,y \in X$, ``$s_x(y) = y$ implies $s_y(x) = x$''. 
\end{enumerate} 
\end{Prop} 

\begin{proof} 
Take any $(v,a) \in X$. 
Recall that a connected component is an orbit of $\mathrm{Inn}(X,s)$. 
Hence, by the definition of $s$, the connected component containing $(v,a)$ satisfies 
\begin{align} 
\mathrm{Inn}(X,s).(v,a) \subset \{ (v,0) , (v,1) \} . 
\end{align} 
This proves (1). 
The assertion (2) also follows easily from the definition of $s$. 
In fact, since edges of $G$ have no orientation, 
one knows that $v \sim w$ if and only if $w \sim v$ for all $v,w \in V(G)$. 
\end{proof} 

We see some more properties of $Q_G$. 
The following properties can be derived from the property mentioned above. 

\begin{Prop} 
Let $(X,s)$ be a quandle, and assume that all connected components consist of at most two points. 
Then we have the following$:$ 
\begin{enumerate} 
\item[$(1)$] 
$(X,s)$ is disconnected if $\# X > 1$$;$ 
\item[$(2)$] 
$\mathrm{Inn}(X,s)$ is abelian. 
\end{enumerate} 
\end{Prop} 

\begin{proof} 
Assume that $\# X > 1$ and $(X,s)$ is connected. 
Then, by the assumption on connected components, we have $\# X = 2$. 
Hence $X$ must be a trivial quandle, which is disconnected. 
This is a contradiction, and hence proves (1). 
In order to show (2), take any $x,y,z \in X$. 
We have only to show that 
\begin{align} 
\label{eq:claim_flat} 
s_x \circ s_y (z) = s_y \circ s_x (z) . 
\end{align} 
If $\mathrm{Inn}(X,s).z = \{ z \}$, then the claim easily follows. 
Assume that $\mathrm{Inn}(X,s).z$ consists of two points, namely $\{ z, z^\prime \}$. 
Note that the symmetry group of $\{ z, z^\prime \}$ is isomorphic to $\Z_2$, which is abelian. 
Therefore, 
$s_x |_{\{ z, z^\prime \}}$ and $s_y |_{\{ z, z^\prime \}}$ commute, 
which completes the proof of (\ref{eq:claim_flat}). 
\end{proof} 

Recall that $\# Q_G \geq 2$. 
Therefore, for every graph $G$, 
we have proved that $Q_G$ is a disconnected quandle and having an abelian inner automorphism group. 

\subsection{Note on abelian extensions} 

In this subsection, 
we mention that the quandles $Q_G$ associated to graphs are abelian extensions. 
For an abelian group $A$, we regard it as an additive group $(A,+)$. 

\begin{Def} 
Let $(X,s)$ be a quandle, and $A$ be an abelian group. 
Then a map $\phi : X \times X \to A$ is called a \textit{quandle $2$-cocycle} if it satisfies 
\begin{enumerate} 
\item[$(1)$] 
$\phi (x,x) = 0$ for any $x \in X$; 
\item[$(2)$] 
$\phi(x,y) - \phi(x,z) + \phi( s_y(x) , z) - \phi ( s_z(x), s_z(y)) = 0$ for any $x,y,z \in X$. 
\end{enumerate} 
\end{Def} 

In the case that $(X,s)$ is a trivial quandle, 
Condition~(2) holds for every map $\phi$. 
By using a quandle $2$-cocycle, one can construct a new quandle, by extending the original quandle $(X,s)$. 

\begin{Prop}[\cite{CENS, CKS}] 
Let $(X,s)$ be a quandle, $A$ be an abelian group, and $\phi : X \times X \to A$ be a quandle $2$-cocycle. 
Then the map $s$ defined by the following gives a quandle structure on $X \times A$$:$ 
\begin{align*} 
s_{(x , a)} (y , b) := ( s_x(y) , b + \phi(v,w) ) \qquad ( x,y \in X , \ a,b \in A ) . 
\end{align*} 
\end{Prop} 

For a simple graph $G = (V(G) , E(G))$, 
we regard the set of vertices $V(G)$ as a trivial quandle. 
As an abelian group, we take $\Z_2$. 
Then the adjacent function $e : V(G) \times V(G) \to \Z_2$ is a quandle $2$-cocycle. 
Note that $e(v,v) = 0$ holds, since the simple graph $G$ contains no loops. 
Hence one can obtain a quandle by this $2$-cocycle, 
which is nothing but the quandle $Q_G$ defined in the previous subsection. 

\section{The quandles associated to vertex-transitive graphs} 
\label{sec5}

Recall that a graph $G = (V(G), E(G))$ is said to be \textit{vertex-transitive} 
if the graph automorphism group $\Aut(G)$ acts transitively on $V(G)$. 
In this section, we show the latter part of Theorem~\ref{thm:1-3}, 
that is, 
the quandle $Q_G$ is homogeneous if and only if $G$ is vertex-transitive. 
A characterization of such quandles is also given. 

\subsection{Some isomorphisms of $Q_G$} 

In this subsection, 
we give some quandle isomorphisms and automorphisms on $Q_G$. 
This gives some relations between graph automorphisms and quandle isomorphisms. 

\begin{Lem} 
\label{lem:isom1}
Let $G_1$ and $G_2$ be graphs, and $\phi : V(G_1) \to V(G_2)$ be a map. 
We consider 
\begin{align} 
\varphi : V(G_1) \times \Z_2 \to V(G_2) \times \Z_2 : (v,a) \mapsto (\phi(v) , a) . 
\end{align} 
Then $\phi$ is a graph isomorphism between $G_1$ and $G_2$ if and only if 
$\varphi$ is a quandle isomorphism between $Q_{G_1}$ and $Q_{G_2}$. 
\end{Lem} 

\begin{proof} 
It is easy to see that $\phi$ is bijective if and only if $\varphi$ is bijective. 
Hence we have only to show that 
$\phi$ is a graph homomorphism if and only if $\varphi$ is a quandle homomorphism. 
One knows that $\phi$ is a graph homomorphism if and only if 
\begin{align} 
\label{eq:graph_homomorphism}
e(v,w) = e (\phi(v) , \phi(w)) \quad (\forall v , w \in V(G_1)) . 
\end{align} 
On the other hand, by definition, $\varphi$ is a quandle homomorphism if and only if 
\begin{align} 
\varphi \circ s_{(v,a)} (w,b) = s_{\varphi (v,a)} \circ \varphi (w,b) 
\end{align} 
holds for any $(v,a), (w,b) \in V(G_1) \times \Z_2$. 
For the both sides of this equation, we have that 
\begin{align} 
\begin{split} 
(\mathrm{LHS}) & = \varphi (w , b + e(v,w)) = (\phi(w) , b + e(v,w)) , \\ 
(\mathrm{RHS}) & = s_{(\phi(v) , a)} (\phi(w) , b) = (\phi(w) , b + e (\phi(v) , \phi(w)) ) . 
\end{split} 
\end{align} 
This proves that $\varphi$ is a quandle isomorphism if and only if (\ref{eq:graph_homomorphism}) holds, 
which completes the proof. 
\end{proof} 

We need another isomorphism of $Q_G$, 
which fixes the graph but interchanges any two points on the same vertex. 

\begin{Lem} 
\label{lem:isom2} 
Let $Q_G := (V(G) \times \Z_2, s)$ be the quandle associated to a graph $G$. 
Then for every $u \in V(G)$, the following $r_u$ is an automorphism of $Q_G$$:$ 
\begin{align} 
r_u : V(G) \times \Z_2 \to V(G) \times \Z_2 : (v,a) \mapsto \left\{ 
\begin{array}{ll} 
(v , a+1) & (v = u) , \\ 
(v , a) & (v \neq u) . 
\end{array} 
\right. 
\end{align} 
\end{Lem} 

\begin{proof} 
It is clear that $r_u$ is bijective. 
In order to show that $r_u$ is a homomorphism, 
take any $(v,a) , (w,b) \in V(G) \times \Z_2$. 
We claim that 
\begin{align} 
r_u \circ s_{(v,a)} (w,b) = s_{r_u (v,a)} \circ r_u (w,b) . 
\end{align} 
For simplicity of the notation, we put 
\begin{align} 
\delta(x , y) := \left\{ 
\begin{array}{ll} 
1 & (x = y) , \\ 
0 & (x \neq y) 
\end{array} \right. 
\end{align} 
for $x,y \in V(G)$. 
Then, for the both sides of the claim, one can see that 
\begin{align} 
\begin{split} 
(\mathrm{LHS}) 
& = r_u (w , b + e(v,w)) 
= (w , b + e(v,w) + \delta(u,w)) , \\ 
(\mathrm{RHS}) 
& = s_{(v, a + \delta(u,v))} (w , b + \delta(u,w)) 
= (w , b + \delta(u,w) + e(v,w)) . 
\end{split} 
\end{align} 
This completes the proof of the claim. 
\end{proof} 

\subsection{The homogeneity} 

In this subsection, 
we prove the main result of this section, 
namely, 
$Q_G$ is homogeneous if and only if $G$ is vertex-transitive. 
We start with a lemma on the connected components. 

\begin{Lem} 
\label{lem:homogeneity} 
Let $Q_G$ be the quandle associated to a graph $G$, 
and assume that $Q_G$ is nontrivial and homogeneous. 
Then all connected components of $Q_G$ consist of two points, 
which are of the form $\{ (u,0) , (u,1) \}$ with $u \in V(G)$. 
\end{Lem} 

\begin{proof} 
Since $Q_G$ is homogeneous, 
it follows from Proposition~\ref{prop:2-14} that all connected components of $Q_G$ have the same cardinality. 
On the other hand, 
one knows from Proposition~\ref{prop:QG1} that 
each connected component consists of at most two points, namely 
\begin{align} 
\mathrm{Inn}(X,s).(u,a) \subset \{ (u,0) , (u,1) \} . 
\end{align} 
If all connected components of $Q_G$ consist of one point, 
then $Q_G$ must be a trivial quandle, which is not the case. 
Hence all connected components consist of two points. 
\end{proof} 

We are in a position to prove the main result of this section. 
The isomorphisms obtained in the previous subsection play fundamental roles. 

\begin{Thm} 
\label{thm:homogeneos} 
Let $G$ be a graph, and $Q_G$ be the quandle associated to $G$. 
Then $Q_G$ is homogeneous if and only if $G$ is vertex-transitive. 
\end{Thm} 

\begin{proof} 
Assume that $G$ is vertex-transitive, 
and show that $\Aut(Q_G)$ acts transitively on $V(G) \times \Z_2$. 
Take any $(v,a) , (w,b) \in V(G) \times \Z_2$. 
By assumption, there exists $\phi \in \Aut(G)$ such that $\phi(v) = w$. 
It then follows from Lemma~\ref{lem:isom1} that 
there exists $\varphi \in \Aut(Q_G)$ such that 
\begin{align} 
\varphi(v,a) = (\phi(v),a) = (w,a) . 
\end{align} 
Therefore, by composing an automorphism given in Lemma~\ref{lem:isom2} if necessary, 
one can find a quandle automorphism which maps $(v,a)$ into $(w,b)$. 

Conversely, assume that $Q_G$ is homogeneous, 
and prove that $G$ is vertex-transitive. 
When $Q_G$ is a trivial quandle, $G$ is an empty graph, which is vertex-transitive. 
Therefore we have only to consider the case that $Q_G$ is nontrivial. 
Take any $v,w \in V(G)$. 
Since $Q_G$ is homogeneous, 
there exists a quandle isomorphism $\varphi \in \Aut(Q_G)$ such that $\varphi(v,0) = (w,0)$. 
Since $Q_G$ is also nontrivial, 
Lemma~\ref{lem:homogeneity} yields that 
all the connected components of $Q_G$ are of the form $\{ (u,0) , (u,1) \}$. 
Therefore, by applying Proposition~\ref{prop:2-14}, we can show that 
there exists a map $\phi : V(G) \to V(G)$ such that 
\begin{align} 
\varphi(u, \ast) = (\phi(u), \ast) \quad (\forall u \in V(G)) . 
\end{align} 
One knows $\phi(v) = w$, and hence we have only to show that $\phi \in \Aut(G)$. 
Here, 
by composing $\varphi$ with the automorphisms given in Lemma~\ref{lem:isom2} if necessary, 
we can assume without loss of generality that 
\begin{align} 
\varphi(u, a) = (\phi(u), a) \quad (\forall u \in V(G) , \ \forall a \in \Z_2) . 
\end{align} 
Lemma~\ref{lem:isom1} thus yields that $\phi$ is a graph isomorphism. 
\end{proof} 

Since there are a lot of vertex-transitive graphs, 
one can construct a lot of homogeneous disconnected quandles with abelian inner automorphism groups. 
We also note that graphs $G$ are not necessary finite. 
If $G$ is infinite, then we can construct an infinite quandle. 

\subsection{A characterization} 

In this subsection, 
we give a characterization of quandles $Q_G$ associated to vertex-transitive graphs $G$. 
Recall that a quandle $Q$ is said to be crossed if 
for all $x,y \in X$, it satisfies ``$s_x(y) = y$ implies $s_y(x) = x$''. 

\begin{Lem} 
\label{lem:0527} 
Let $Q$ be a crossed quandle, 
and let $\{ v_0 , v_1 \}$ and $\{ w_0 , w_1 \}$  be connected components of $Q$. 
Then the following are mutually equivalent$:$ 
\begin{enumerate} 
\item[$(1)$] 
$s_{v_0} (w_0) = w_0$$;$ \quad $(2)$ $s_{v_0} (w_1) = w_1$$;$ 
\item[$(3)$] 
$s_{v_1} (w_0) = w_0$$;$ \quad $(4)$ $s_{v_1} (w_1) = w_1$$;$ 
\item[$(5)$] 
$s_{w_0} (v_0) = v_0$$;$ \quad $(6)$ $s_{w_0} (v_1) = v_1$$;$ 
\item[$(7)$] 
$s_{w_1} (v_0) = v_0$$;$ \quad $(8)$ $s_{w_1} (v_1) = v_1$. 
\end{enumerate} 
\end{Lem} 

\begin{proof} 
Since $Q$ is crossed, one has the following four equivalences: 
\begin{align} 
(1) \Leftrightarrow (5) , \qquad 
(2) \Leftrightarrow (7) , \qquad 
(3) \Leftrightarrow (6) , \qquad 
(4) \Leftrightarrow (8) . 
\end{align} 
Since $\{ w_0 , w_1 \}$ is a connected component, 
it is preserved by $s_{v_i}$. 
Therefore, $s_{v_i}$ gives a bijection from $\{ w_0 , w_1 \}$ onto itself. 
This yields that 
\begin{align} 
(1) \Leftrightarrow (2) , \qquad 
(3) \Leftrightarrow (4) . 
\end{align} 
Similarly, 
since $s_{w_i}$ gives a bijection from $\{ v_0 , v_1 \}$ onto itself, 
one has that 
\begin{align} 
(5) \Leftrightarrow (6) , \qquad 
(7) \Leftrightarrow (8) . 
\end{align} 
This completes the proof. 
\end{proof} 

By applying this lemma, 
we can give a sufficient condition for a quandle $Q$ to be constructed from a graph $G$. 

\begin{Lem} 
\label{lem:characterization2} 
Let $Q$ be a crossed quandle, 
and assume that all connected components of $Q$ consist of two points. 
Then there exists a graph $G$ such that $Q$ is isomorphic to $Q_G$. 
\end{Lem} 

\begin{proof} 
First of all we construct a graph $G$. 
Let $V$ be the set of all connected components of $Q$, 
which will be the set of vertices. 
We write it as 
\begin{align} 
V = \{ v^\lambda \mid \lambda \in \Lambda \} . 
\end{align} 
The set of edges are defined as follows. 
By assumption, each $v^\lambda$ consists of two points, namely 
\begin{align} 
v^\lambda = \{ v^\lambda_0 , v^\lambda_1 \} . 
\end{align} 
We join $v^\lambda$ and $v^\mu$ by an edge if $s_{v^\lambda_0}$ acts nontrivially on $v^\mu$. 
This defines a graph $G = (V,E)$, 
which is well-defined by Lemma~\ref{lem:0527}. 
Let $e : V \times V \to \Z_2$ be the adjacent function of the graph $G$, 
and denote by $s^Q$ the quandle structure of $Q$. 
Then one has that 
\begin{align} 
\label{eq:0601} 
s^Q_{v^\lambda_a} (v^\mu_b) = v^\mu_{b + e (v^\lambda , v^\mu)} . 
\end{align} 
Let $Q_G = (V \times \Z_2 , s)$ be the quandle associated to the graph $G$. 
It remains to show that $Q_G$ is isomorphic to $Q$. 
We define a map 
\begin{align} 
f : V \times \Z_2 \to Q : ( v^\lambda , a ) \mapsto v^\lambda_a . 
\end{align} 
This is clearly bijective, 
and also we can show that $f$ is a homomorphism. 
In fact, by definition of the quandle structure $s$, one has 
\begin{align} 
f \circ s_{(v^\lambda , a)} (v^\mu , b) 
= f (v^\mu , b + e (v^\lambda , v^\mu) ) 
= v^\mu_{b + e (v^\lambda , v^\mu)} . 
\end{align} 
Furthermore, 
it follows from (\ref{eq:0601}) that 
\begin{align} 
s^Q_{f( v^\lambda , a )} \circ f (v^\mu , b) 
= s^Q_{ v^\lambda_a } ( v^\mu_b ) 
= v^\mu_{b + e (v^\lambda , v^\mu)} . 
\end{align} 
This proves that $f$ is a homomorphism, and hence an isomorphism. 
\end{proof} 

We now give a characterization for quandles associated to vertex-transitive graphs. 
Note that we are only interested in the case of nontrivial quandles. 

\begin{Prop} 
\label{prop:characterization} 
Let $Q$ be a quandle. 
Then, it is isomorphic to $Q_G$ for some nonempty vertex-transitive graph $G$ if and only if it satisfies 
\begin{enumerate} 
\item[$(1)$] 
all connected components of $Q$ consist of two points$;$ 
\item[$(2)$] 
$Q$ is crossed$;$ 
\item[$(3)$] 
$Q$ is homogeneous. 
\end{enumerate} 
\end{Prop} 

\begin{proof} 
Let $G$ be a nonempty vertex-transitive graph. 
Then, $Q_G$ satisfies 
(1) by Lemma~\ref{lem:homogeneity}, 
(2) by Proposition~\ref{prop:QG1}, 
and (3) by Theorem~\ref{thm:homogeneos}. 

It remains to show the ``if''-part. 
Assume that $Q$ satisfies the above three conditions. 
By Conditions~(1) and (2), 
Lemma~\ref{lem:characterization2} yields that 
there exists a graph $G$ such that $Q$ is isomorphic to $Q_G$. 
Furthermore, by Condition~(3), 
Theorem~\ref{thm:homogeneos} yields that $G$ is vertex-transitive. 
Finally, $G$ is nonempty since $Q \cong Q_G$ is nontrivial. 
This completes the proof. 
\end{proof} 

\section*{Acknowledgements} 

The authors would like to thank Mika Nagashiki for a careful reading of the manuscript, 
and Takayuki Okuda for kind advice and suggestions, especially on graph theory. 
They are also grateful to Akira Kubo, Yuichiro Taketomi, Fumiaki Akase, and Kokoro Tanaka for helpful comments.

\end{document}